\documentclass[10pt,a4paper]{amsart}
\usepackage{amsmath,amsthm,amstext,amssymb,a4,stmaryrd,hyperref}
\usepackage{graphics,epsfig}
\usepackage{psfrag}
\usepackage{graphicx}
\usepackage[T1]{fontenc}

\newcommand{\N}{\ensuremath{\mathbb{N}}}
\newcommand{\R}{\ensuremath{\mathbb{R}}}

\renewcommand{\.}{{\cdot}}

\newcommand{\D}{\mathrm{d}}

\theoremstyle{plain}
\newtheorem{Thm}{Theorem}[]
\newtheorem{Lem}{Lemma}[]
\newtheorem{corollary}{Corollary}[]

\theoremstyle{remark}
\newtheorem*{remark}{Remark}

\theoremstyle{definition}
\newtheorem{Def}{Definition}[]
\newtheorem{example}{Example}[]

\begin{document}

\author[]{A. Dirmeier}
\address{Department of Mathematics, Technische Universit{\"a}t Berlin,
Str.~d.~17.~Juni 136, 10623 Berlin, Germany}

\title[]{Growth conditions for conformal transformations preserving Riemannian completeness}

\begin{abstract}
For a complete Riemannian metric, a pointwise conformal transformation may lead to a complete or incomplete transformed Riemannian metric, depending on the behavior of the conformal factor. We establish conditions on the growth of the conformal factor towards the infinity of the Riemannian manifold, such that the conformally transformed Riemannian metric remains complete. 
\end{abstract}

\maketitle

%\tableofcontents

In 1961 Nomizu and Ozeki \cite{Nomizu1961} established the result that every manifold $M$, satisfying the second axiom of countability, admits complete and incomplete Riemannian metrics. Moreover these are connected by conformal transformations. Hence, for every incomplete Riemannian metric, we can find a conformal factor to make it complete and vice versa. Subsequently, it has been established (cf.~\cite{Fegan1978} and \cite{Morrow1970}) that the complete and the incomplete Riemannian metrics are dense in the space of Riemannian metrics over the manifold $M$. Based on \cite{Fegan1978}, it possible to establish a partial order for Riemannian metrics $g,h$ on $M$ by 
\[
 g\leq h\ :\Leftrightarrow\ g_x(v,v)\leq h_x(v,v)
\]
for all $x\in M$ and $v\in T_xM$. Now if $g$ is complete and $g\leq h$, then $h$ is complete. As usual, we will call a Riemannian metric $g$ on $M$ or the Riemannian manifold $(M,g)$ complete if $M$ is complete with respect to the distance function $d_g(\.,\.)$ associated to $g$. The aim of the present paper is to establish conditions on a conformal factor $A\colon M\to(0,\infty)$, which transforms a given complete Riemannian metric $g$ to $g^{\prime}=\frac{g}{A^2}$, such that $g^{\prime}$ is still complete. Following \cite{Nomizu1961}, a conformal factor which makes $g^{\prime}$ incomplete is easy to find. Take for example exponential growth of $A$ towards $g$-infinity, then $g^{\prime}$ is incomplete (if $M$ is non-compact). But the exact upper bound on the growth of $A$, where $g^{\prime}$ changes from complete to incomplete has hitherto not been established. This is the content of Thm.~\ref{theorem} below.

The author is indebted to Prof.~P.~B\'erard from the Universit\'e de Grenoble for pointing out an error in an earlier version of this paper, as well as in \cite{Dirmeier2010}, as well as to Prof.~E.~Caponio from the Politecnico di Bari for a fruitful discussion on the corrections to \cite{Dirmeier2010}, which also yielded some corrections of this paper. 

We will start with some definitions. All functions are assumed to be smooth, considering their use as conformal factors below, although for Def.~\ref{definition} and Lem.~\ref{lem1} measurable functions suffice. Furthermore, we assume all curves to be regular.

Let $M$ be a manifold. We will assume all manifolds to be non-compact. A ray $\alpha\colon[0,b)\to M$ (or a curve $\tilde{\alpha}\colon (a,b)\to M$ for some $a<0<b$ such that $\alpha=\tilde{\alpha}|_{[0,b)}$ is a ray) with $0<b\leq\infty$ will be called \textit{escaping (to infinity)} on $M$ if there is a sequence $\{t_n\}_{n\in\N}\subset[0,b)$ with $t_n\to b$ as $n\to\infty$, such that $\alpha(t_n)\to\infty$ as $n\to\infty$ in the following sense: there is an exhaustion of $M$ by compact sets $\{K_n\}_{n\in\N_0}\subset\mathcal{P}(M)$ (the power set of $M$), $M=\bigcup_{n\in\N_0}K_n$, $K_0=\{\alpha(0)\}$, $K_{n}\subset\mathring{K}_{n+1}$ for all $n\in\N_0$, such that $\alpha(t_n)\in K_n\setminus K_{n-1}$ for all $n\in\N$. 

\begin{Def}\label{definition}
Let $(M,g)$ be a complete Riemannian manifold.
\begin{itemize}
 \item[(i)] A function $f\colon M\to(0,\infty)$ \textit{grows at most linearly (sublinearly) towards $g$-infinity} on $(M,g)$ if for all fixed $x_0\in M$ there are constants $c_1,c_2>0$, such that for all $x\in M$, $f(x)\leq c_1d_g(x_0,x)+c_2$ holds.
 \item[(ii)] A function $f\colon M\to(0,\infty)$ \textit{grows superlinearly towards $g$-infinity} on $(M,g)$ if there are constants $\epsilon,c_1,c_2>0$, such that for all fixed $x_0\in M$ and all $x\in M$, $f(x)\geq c_1d_g(x_0,x)^{1+\epsilon}+c_2$ holds.
 \item[(iii)] A function $f\colon M\to(0,\infty)$ will be called \textit{$L^1$ on an escaping curve w.r.t.~$g$} if there is $x_0\in M$ and an escaping ray $\gamma\colon[0,\infty)\to M$ with $\gamma(0)=x_0$, such that \[\int_0^{\infty}(f\circ\gamma)(s)\sqrt{g_{\gamma(s)}(\dot{\gamma}(s),\dot{\gamma}(s))}\D s<\infty.\] Obviously, this condition is independent of the parametrization of $\gamma$, and without loss of generality, we can assume $\gamma$ to be parametrized by arc length, such that $g(\dot{\gamma},\dot{\gamma})=1$, and we have $\int_0^{\infty}(f\circ\gamma)(s)\D s<\infty$, i.e., $f\circ\gamma\in L^1([0,\infty))$. 
\end{itemize}
\end{Def}

Inspecting item (iii), we observe that the condition is equivalent to state that the length of the escaping ray in the conformally transformed metric $f^2g$ on $M$ is finite. As it is known that a Riemannian metric is complete if and only if the length of any escaping curve is infinite, the connection to completeness becomes obvious. But stating the condition in terms of growth of the conformal factor, instead of in terms of conformally transformed curve lengths, allows to compare this condition to the linear growth conditions in items (i) and (ii). The precise relation is clarified in the following

\begin{Lem}\label{lem1}
If $f\colon M\to(0,\infty)$ grows superlinearly towards $g$-infinity on a complete Riemannian manifold $(M,g)$, then $\frac{1}{f}\colon M\to(0,\infty)$ is $L^1$ on \emph{all} escaping $g$-geodesic rays. If $f\colon M\to(0,\infty)$ grows at most linearly towards $g$-infinity on a complete Riemannian manifold $(M,g)$, then $\frac{1}{f}\colon M\to(0,\infty)$ is \emph{not} $L^1$ on any escaping curve w.r.t.~$g$.    
\end{Lem}

\begin{proof}
Let $x_0\in M$ be a fixed point and $\gamma\colon[0,\infty)\to M$ any escaping $g$-geodesic ray with $\gamma(0)=x_0$, parametrized by arc length. Assume $f$ to grow superlinearly towards $g$-infinity, thus there are constants $\epsilon,c_1,c_2>0$ such that
\[
 \frac{1}{f(x)}\leq\frac{1}{c_1d_g(x_0,x)^{1+\epsilon}+c_2},
\]
for all $x_0,x\in M$. Hence, we have for all $s\in[0,\infty)$
\[
 \frac{1}{(f\circ\gamma)(s)}\leq\frac{1}{c_1d_g(x_0,\gamma(s))^{1+\epsilon}+c_2}=\frac{1}{c_1 s^{1+\epsilon}+c_2}\in L^1([0,\infty))
\]
as $\epsilon,c_1,c_2>0$.

Assume now that $f$ grows at most linearly towards $g$-infinity on $M$. Hence, for all $x_0\in M$, there are constants $c_1,c_2>0$ such that
\[
 \frac{1}{f(x)}\geq\frac{1}{c_1d_g(x_0,x)+c_2},
\]
for all $x\in M$. Let $\gamma\colon[0,\infty)\to M$ be any escaping ray emanating at $x_0\in M$ and parametrized by arc length. Then we have $d_g(x_0,\gamma(s))\leq s$ for all $s\in[0,\infty)$ and thus
\[
 \frac{1}{(f\circ\gamma)(s)}\geq\frac{1}{c_1d_g(x_0,\gamma(s))+c_2}\geq\frac{1}{c_1 s+c_2}.
\]
This implies $\int_{[0,\infty)}\frac{1}{f\circ\gamma}=\infty$ as $\int_0^{\infty}\frac{\D s}{c_1 s+c_2}=\infty$ for all $c_1,c_2>0$.
\end{proof}

Now we are ready to state the main theorem.

\begin{Thm}\label{theorem}
Let $(M,g)$ be a non-compact and complete Riemannian manifold and $A\colon M\to(0,\infty)$ a positive function. We denote a conformally transformed metric on $M$ by $g^\prime=\frac{g}{A^2}$. Then $(M,g^\prime)$ is complete if and only if $\frac{1}{A}\colon M\to(0,\infty)$ is not $L^1$ on any escaping curve w.r.t.~$g$. Moreover, if $A$ grows at most linearly towards $g$-infinity on $M$, then $(M,g^\prime)$ is complete and if $A$ grows superlinearly towards $g$-infinity on $M$, then $g^\prime$ is a bounded metric on $M$, particularly $(M,g^\prime)$ is incomplete. 
\end{Thm}

\begin{proof}
We will show the following statement: $\frac{1}{A}$ is $L^1$ on an escaping curve w.r.t.~$g$ if and only if $(M,g^\prime)$ is incomplete. The first remaining statement then follows easily from Lem.~\ref{lem1}.

Assume first that there is $x_0\in M$ and a ray $\gamma\colon[0,\infty)\to M$ with $\gamma(0)=x_0$ escaping to infinity, such that $\frac{1}{A}$ is $L^1$ on $\gamma$ w.r.t.~$g$. We can parametrize $\gamma$ by $g$-arc length, i.e., we have $g(\dot{\gamma},\dot{\gamma})=1$, and hence some constant $B<\infty$ such that $\int_0^\infty\frac{\D s}{(A\circ\gamma)(s)}=B$. Take any sequence $\{s_n\}_{n\in\N}\subset[0,\infty)$, with $s_n\to\infty$ and $\gamma(s_n)\to\infty$ as $n\to\infty$. Then we compute the distance between $x_0$ and each $\gamma(s_n)$ in the conformally transformed metric $g^\prime$:
\[
 d_{g^\prime}(x_0,\gamma(s_n))\leq\int_0^{s_n}\frac{\sqrt{g(\dot{\gamma}(s),\dot{\gamma}(s))}}{A(\gamma(s))}\D s\leq\int_0^{\infty}\frac{1}{(A\circ\gamma)(s)}\D s=B. 
\]
Hence, the sequence $\{\gamma(s_n)\}_{n\in\N}$ is contained in a closed and bounded $g^\prime$-ball of radius $B$ about $x_0$. But by the definition of escaping curves the sequence $\{\gamma(s_n)\}_{n\in\N}$ has no convergent subsequence, thus the closed and bounded $g^\prime$-ball of radius $B$ about $x_0$ is not compact and, therefore, $(M,g^\prime)$ is incomplete by the Hopf--Rinow theorem.

Assume now that $(M,g^\prime)$ is incomplete, hence there is a point $x_0\in M$ and a $g^\prime$-geodesic ray $\gamma\colon[0,b)\to M$ emanating from $x_0$ that is not extendible to the parameter value $b$. As a $g^\prime$-geodesic is parametrized to unit $g^\prime$-velocity, we conclude the length of $\gamma$ to be $b<\infty$. But obviously $\gamma$ escapes to infinity, as there exists no point $\gamma(b)\in M$. Now reparametrize $\gamma$ to unit $g$-velocity, i.e., $g(\dot{\gamma},\dot{\gamma})=1$, then we get $\gamma\colon[0,\infty)\to M$ as $g$ is assumed complete. We compute
\[
 \infty>b=\int_0^\infty\frac{\sqrt{g(\dot{\gamma}(s),\dot{\gamma}(s))}}{A(\gamma(s))}\D s=\int_0^{\infty}\frac{1}{(A\circ\gamma)(s)}\D s.
\]
Hence, $\frac{1}{A}$ is $L^1$ on the escaping curve $\gamma$ w.r.t.~$g$.   

As $\frac{1}{A}$ is $L^1$ on all escaping $g$-geodesic rays if $A$ grows superlinearly towards $g$-infinity on $M$ by Lem.~\ref{lem1}, we certainly have in this case that $g^\prime$ is incomplete. Furthermore, computing the distance of a fixed $x_0\in M$ to any $x\in M$ in the $g^\prime$ distance we get for some finite constant $r(\epsilon,c_1,c_2)<\infty$
\[
 d_{g^\prime}(x_0,x)\leq\int_0^\infty\frac{\D s}{c_1s^{1+\epsilon}+c_2}=r(\epsilon,c_1,c_2).
\]
Hence, $g^\prime$ is bounded as now $d_{g^\prime}(x,y)\leq d_{g^\prime}(x_0,x)+d_{g^\prime}(x_0,y)=2r$ holds for all $x,y\in M$.  
\end{proof}

As a consequence of theorem above, we can now establish conditions for the completeness of a Riemannian metric $g=h-s$, emerging from a complete Riemannian metric $h$ and a non-negative, symmetric $(0,2)$-tensor field $s$ (i.e., $s_x(v,v)\geq0$ for all $x\in M$ and all $v\in T_xM$) on a manifold $M$. Obviously, $g$ is a non-degenerate Riemannian metric if $\frac{s_x(v,v)}{h_x(v,v)}<1$ for all $x\in M$ and all $v\in T_xM\setminus\{0\}$. We can now define an $h$-norm for $(0,2)$-tensor fields on $M$. For the tensor field $s$, this norm is given at some point $x\in M$ by
\[
 \interleave s\interleave^h_x=\sup_{v\in T_xM\setminus\{0\}}\sqrt{\frac{s_x(v,v)}{h_x(v,v)}}.
\]
Then $g$ is complete if $\sup_{x\in M}\interleave s\interleave^h_x<1$, because in this case $h\leq g$. But if $\sup_{x\in M}\interleave s\interleave^h_x=1$, the metric $g$ can be complete anyway if $\interleave s\interleave^h_x$ obeys certain growth conditions, which can be inferred from Thm.~\ref{theorem}. 

\begin{corollary}\label{cor1}
Let $(M,h)$ be a complete Riemannian manifold and $s$ a non-negative, symmetric $(0,2)$-tensor field $M$. Let the symmetric $(0,2)$-tensor field $g$, given by $g=h-s$, be a Riemannian metric for all $x\in M$. Assume $\sup_{x\in M}\interleave s\interleave^h_x=1$, then $g$ is complete if the function
\[
 \sqrt{1-\left(\interleave s\interleave^h_x\right)^2}
\]
is not $L^1$ on any escaping curve w.r.t.~$h$ on $M$, and particularly if for all fixed $x_0\in M$ there are constants $c_1,c_2>0$, such that
\[
 \left(\interleave s\interleave^h_x\right)^2\leq1-\frac{1}{(c_1d_h(x,x_0)+c_2)^2}
\]
holds for all $x\in M$.
\end{corollary}
 
\begin{proof}
For all $x\in M$ and $v\in T_xM$, we compute
\[
 g_x(v,v)=h_x(v,v)(1-\frac{s_x(v,v)}{h_x(v,v)})\geq h_x(v,v)(1-(\interleave s\interleave^h_x)^2)=\frac{h_x(v,v)}{\frac{1}{1-(\interleave s\interleave^h_x)^2}}=:h^{\prime}_x(v,v).
\]
So $g$ is complete if $h^{\prime}$ is complete, and by Thm.~\ref{theorem}, for complete $h$, the metric $h^{\prime}$ is complete if and only if $\sqrt{1-\left(\interleave s\interleave^h_x\right)^2}$ is not $L^1$ on any escaping curve w.r.t.~$h$ on $M$.

Furthermore, $h^\prime$ is also complete if $\frac{1}{\sqrt{1-\left(\interleave s\interleave^h_x\right)^2}}$ grows at most linearly towards $h$-infinity on $M$. By item (i) in Def.~\ref{definition} this means that for all fixed $x_0\in M$ there are $c_1,c_2>0$, such that
\[
 \frac{1}{\sqrt{1-\left(\interleave s\interleave^h_x\right)^2}}\leq c_1 d_h(x_0,x) + c_2
\]
holds for all $x\in M$ and the result follows. 
\end{proof}

\begin{remark}
A special case of the Corollary is at hand if the tensor field $s$ is given by $s=\beta\otimes\beta$, with $\beta$ being a one-form on the manifold $M$. In this case the tensor norm $\interleave\cdot\interleave$ can be replaced by the usual norm for one-forms given by
\[
 \|\beta\|^h_x=\sup_{v\in T_xM\setminus0}\sqrt{\frac{(\beta_x(v))^2}{h_x(v,v)}},
\] 
and the corollary holds in an analogue version.
\end{remark}

\begin{example}
For an instructive example we can look at the flat metric
\[
 \delta=dr^2+r^2d\Omega^2
\]
on the punctured Euclidian space $\R^3\setminus\{0\}$, with $d\Omega^2=d\theta^2+\sin^2\theta d\varphi^2$ being the usual metric on the two-sphere $S^2$. This metric is incomplete because for any fixed angle $\Omega_0$ we can find a radial line segment $x(t)=(-t,\Omega_0)$ with $t\in[-1,0)$, approaching the removed origin from the unit sphere. For the flat metric $\delta$ this is a geodesic arc, not extendible to $t=0$. So obviously with $\dot{x}=(-1,0)$ we have
\[
 d_{\delta}\left((1,\Omega_0),(0+\epsilon,\Omega_0)\right)=\int_{-1}^{0-\epsilon}\sqrt{\delta(\dot{x},\dot{x})}dt=
 \int_{-1}^{0-\epsilon}dt=1-\epsilon<\infty.
\]
If we now consider the conformally transformed metric 
\[
 \tilde{\delta}=\frac{\delta}{r^2}=\frac{dr^2}{r^2}+d\Omega^2,
\]
we observe that this metric is complete on $\R^3\setminus\{0\}$. This is obvious by choosing a new radial coordinate $\rho=\ln r$. For $r\in(0,\infty)$ and $r=1$, we now have $\rho\in(-\infty,\infty)$ and $\rho=0$. Thus the curve $y(t)=(-t,\Omega_0)$ with $t\in[0,\infty)$ and fixed angle $\Omega_0$ in the new coordinates is a geodesic arc for $\tilde{\delta}$. This geodesic approaches negative $\rho$-infinity---which corresponds to $r=0$ in the old coordinates---for the curve parameter $t\to\infty$. Thus one could say that the conformal transformation moved the origin to infinite distance. Clearly, we have for $\dot{y}=(-1,0)$
 \[
 d_{\tilde{\delta}}\left((0,\Omega_0),(\infty,\Omega_0)\right)=\int_{0}^{\infty}\sqrt{\tilde{\delta}(\dot{y},\dot{y})}dt=
 \int_{0}^{\infty}dt=\infty
\]
so that $\tilde{\delta}$ is complete. Moreover, one can infer from these considerations, that $(\R^3\setminus\{0\},\delta)$ is conformally equivalent to $(\R\times S^2,\tilde{\delta})$. If we now impose a conformal transformation 
\[
 \tilde{\tilde{\delta}}=\frac{\tilde{\delta}}{(A(\rho))^2}
\]
depending on the radial coordinate $\rho$ (resp.~$r$) only, we have for the radial distance with fixed angle $\Omega_0$
\[
 d_{\tilde{\delta}}(\rho_0,\rho)=|\rho-\rho_0|.
\]
So setting $\rho_0=0$ (resp.~$r=1$), we get from Thm.~\ref{theorem} that provided
\[
 A(\rho)\leq c_1|\rho|+c_2
\]
holds, $A$ grows at most linearly towards infinity and $\tilde{\tilde{\delta}}$ is complete. In the coordinate $r$, this means that a conformal factor $A(r)$ imposed on $\tilde{\delta}$ may at most grow by
\[
 A(r)\leq c_1|\ln r|+c_2
\]
towards the removed origin $r=0$ and the metric
\[
 \tilde{\delta}^{\prime}=\frac{\frac{dr^2}{r^2}+d\Omega^2}{(c_1|\ln r|+c_2)^2}
\]
is complete. As an example for Cor.~\ref{cor1} we consider the metric
\[
 h=\tilde{\delta}-\beta^2=\left(1-\frac{\rho^2}{\rho^2+1}\right)d\rho^2+d\Omega^2 
\]
on $\R\times S^2$, with $\beta=\frac{\rho}{\sqrt{\rho^2+1}}d\rho$. Now for all $x_0=(\rho_0,\Omega_0)\in\R^3\setminus\{0\}$ we choose $c_1=1$ and $c_2=\sqrt{\rho_0^2+1}$ in the corollary above. If now
\[
 (\|\beta\|^{\tilde{\delta}}_{\rho})^2=\frac{\rho^2}{\rho^2+1}\leq1-\frac{1}{(d_{\tilde{\delta}}(\rho_0,\rho)+\sqrt{\rho_0^2+1})^2}
\]
holds for all $\rho,\rho_0\in\R$, the metric $h$ is complete. The inequality above can be straightforwardly transformed to the equivalent inequality
\[
 |\rho-\rho_0|\sqrt{\rho_0^2+1}\geq\rho_0(\rho-\rho_0),
\]
by using $d_{\tilde{\delta}}(\rho_0,\rho)=|\rho-\rho_0|$. This inequality certainly holds for all $\rho,\rho_0\in\R$. 
\end{example}


\begin{thebibliography}{1}
\providecommand{\url}[1]{\texttt{#1}}
\providecommand{\urlprefix}{URL }
\providecommand{\eprint}[2][]{\url{#2}}

\bibitem{Dirmeier2010}
A.~Dirmeier, M.~Plaue, M.~Scherfner, \emph{Growth Conditions, Riemannian
  Completeness and Lorentzian Causality}, J.~Geom.~Phys., \textbf{62(3)} (2012), 604-612.

\bibitem{Fegan1978}
H.~D. Fegan, R.~S. Millman, \emph{Quadrants of Riemannian Metrics}, Michigan
  Math. J., \textbf{25} (1978), 3--7.

\bibitem{Morrow1970}
J.~A. Morrow, \emph{The denseness of complete Riemannian metrics}, J. Diff.
  Geom., \textbf{4(2)} (1970), 225--226.

\bibitem{Nomizu1961}
K.~Nomizu, H.~Ozeki, \emph{The existence of complete Riemannian metrics}, Proc.
  Amer. Math. Soc., \textbf{12} (1961), 889--891.

\end{thebibliography}
\end{document}